\documentclass{enarticle}

\newcommand{\bfX}{\mathbf{X}}
\newcommand{\bfY}{\mathbf{Y}}

\newcommand{\tS}{\widetilde{S}}
\newcommand{\tB}{\widetilde{B}}

\title{Random walks veering left}
\author{Raoul \textsc{Normand} \& B\'alint \textsc{Vir\'ag}}
\location{Department of Mathematics - University of Toronto}
\email{rnormand@math.utoronto.ca \& balint@math.toronto.edu}

\begin{document}

\maketitle

\begin{abstract}
We study coupled random walks in the plane such that, at each step, the walks change direction by a uniform random angle plus an extra deterministic angle $\th$. We compute the Hausdorff dimension of the $\th$ for which the walk has an unusual behavior. This model is related to a study of the spectral measure of some random matrices.
\end{abstract}

\begin{scriptsize}

\end{scriptsize}

\section{Introduction}

\subsection{Model}

The goal of this paper is to study random walks in the complex plane. The simplest one is constructed by turning at each step by a uniform angle, and taking a step of length 1. Now, we want to have a whole family of coupled random walks indexed by $\th \in [0,2 \pi)$, and to this end, we just perform the following simple operation: if, at step $n$, the initial walk turns an angle $\phi_n$, then the one indexed by $\th$ turns an angle $\phi_n + \th$, see Figure \ref{fig:walk}. This does not require any additional randomness but, as we shall see, these walks can have quite different behavior.

\begin{figure}[hbt]
\centering
\includegraphics[width=0.7\columnwidth]{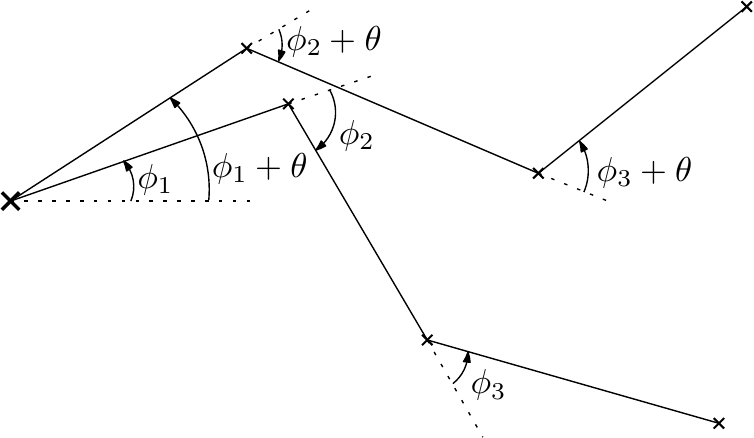}
\caption{Three first steps of two coupled random walks}
\label{fig:walk}
\end{figure}

Intuitively, for two close $\th$, the corresponding walks will remain close for a long time, then spread apart and have quite independent behaviors. A natural guess and, as we shall see, a true one in some sense, is that this happens at a time of order roughly $1/\th$.

All the walks have the same law, and thus, after a time $n$, they are at a distance of order $\sqrt{n}$ from the origin. Now, informally, amongst the $n$ roughly independent walks with $\th = 0, 2 \pi / n,\dots, 2 \pi (n-1)/n$, a small amount may have an exceptional behavior, e.g. be much farther away from the origin. We are interested in the set of these $\th$ for which the walk is exceptionally far. It turns out that the correct threshold is of order $\sqrt{n \ln n}$, and we shall compute the Hausdorff dimension of the set of angles for which the walk is beyond this threshold infinitely many times.

\subsection{Notation and result}

Before giving the motivation for this model, let us introduce some notation and our main result. If $(\th_j)_{j \geq 1}$ is a family of i.i.d. uniform variables in $[0,2 \pi)$, then the model can be written as 
\begin{align*}
S_0(\th) & = 0 \\
S_1(\th) & = e^{i (\th_1 + \th)} = e^{i \th} e^{i \th_1} \\
S_2(\th) & = e^{i (\th_1 + \th)} + e^{i (\th_1 + \th)} e^{i (\th_2 + \th)} = e^{i \th} e^{i \th_1} + e^{2 i \th} e^{i (\th_1 + \th_2)} \\
S_3(\th) & =  e^{i \th} e^{i \th_1} + e^{2 i \th} e^{i (\th_1 + \th_2)} + e^{2 i \th} e^{i (\th_1 + \th_2)} e^{i (\th_3 + \th)} = e^{i \th} e^{i \th_1} + e^{2 i \th} e^{i (\th_1 + \th_2)} + e^{3 i \th} e^{i (\th_1 + \th_2 + \th_3)}
\end{align*}
and so on. But clearly, the variables $e^{i (\th_1 + \dots + \th_j)}$ for $j \geq 1$ are all uniform rotations and are independent, so we might actually replace $\th_1 + \dots + \th_j$ by $\th_j$. More generally, we can choose the length of each step to be random, independent from the direction, while still keeping the latter uniform. The only assumption that we will need on the modulus of these variables is that it has an exponential moment. Hence, let us forget about $(\th_j)$, and fix once and for all the following notation. 

\begin{itemize}
\item Let $(U_j)_{j \geq 1}$ be an i.i.d. sequence of complex random variables with a rotationally symmetric law, and with an exponential moment, i.e. such that
\[
\E ( \exp \k |U_1| ) < \pinf
\]
for some $\k > 0$. Write $\s^2 := \E(\Re(U_1)^2) > 0$.
\item Denote $\cN_{\C}(0, \rho^2)$ the distribution of a complex Gaussian variable with covariance matrix $\rho I_2$. 
\item Let $(G_j)_{j \geq 1}$ an i.i.d. sequence of standard complex Gaussian variables, i.e. with law $\cN_{\C}(0,1)$.
\item For $\th \in [0,1)$, define
\[
S_n(\th) = U_1 e^{2 i \pi \th} + \dots + U_n e^{2 i n \pi \th}, \quad S_n = S_n(0), \quad \tS_n(\th) = \frac{S_n}{\s \sqrt{2 n}}
\]
and
\[
B_n(\th) = G_1 e^{2 i \pi \th} + \dots + G_n e^{2 i n \pi \th}, \quad B_n = B_n(0), \quad \tB_n(\th) = \frac{B_n}{\sqrt{2 n}}.
\]
\item Fix a constant $\a > 0$, and a threshold $\phi(n) = \sqrt{2 \a n \log n}$.
\end{itemize}

We are interested in the Hausdorff dimension of the set of exceptional angles
\[
\cD_{\a} = \{\th \in [0,1), \, |S_n(\th)| > \s \phi(n) \; \text{i.o.} \}.
\]
We call this angles ``exceptional'' because, by the Central Limit Theorem, $|S_n|$ should be of order $\sqrt{n}$, so that, for every $\th \in [0,1)$, almost surely $\th \notin \cD_{\a}$. However we expect that a.s., $\cD_{\a}$ is not empty, and even more, that its Hausdorff dimension is nontrivial. This relies on the following computation: loosely, $S_n(\th)/(\s \sqrt{n})$ should be close to a standard complex Gaussian variable, and thus
\[
\P \left ( |S_n(\th)| > \s \phi(n) \right ) \approx \P \left ( |G_1| > \sqrt{2 \a \log n} \right ) = \frac{1}{n^{\a}}.
\]
The last equality stems from the fact that $|G_1|^2$ has a $\chi^2$ distribution with two degrees of freedom, whose c.d.f. is $x \mapsto 1- e^{-x/2}$. This decrease as a power of $n$ is precisely what one would expect to obtain a nontrivial Hausdorff dimension, and explains why we choose such $\phi(n)$. We will prove the following.

\begin{thm} \label{th:dimSn}
Almost surely,
\[
\dim \cD_{\a} = (1 - \a) \vee 0.
\]
\end{thm}

The incentive for distinguishing the Gaussian case is that we shall first prove the result for a Gaussian random walk, since, as usual, computations are easier in this case. Define its set of exceptional angles by
\[
\cD'_{\a} = \{ \th \in [0,1), \; |B_n(\th)| > \phi(n) \; \text{i.o.} \}.
\]
The corresponding result is that, almost surely,
\begin{equation} \label{eq:dimBn}
\dim \cD'_{\a} = (1 - \a) \vee 0.
\end{equation}

This question is related to other results about dynamical random walks, where the steps are refreshed independently after an exponential time. For instance, \cite{BenjaminiRW} studies properties of random walks which are invariant or not under this change, as well as Hausdorff dimensions of exceptional times. The somehow surprising difference here is that we can obtain nontrivial dimensions without needing extra randomness.

Let us add a couple words concerning the notation used. In the whole text, for $q > 1$, we will write $q^n$ where we mean $\lfloor q^n \rfloor$; the reader could readily fill in the occasional gaps and convince themselves of the innocuousness of such treatment. We shall also always keep in the subtext that the big $O$ notation is uniform, in that the constant hidden inside is the same for every $n$ and all $\th$ in the considered interval (see in particular Lemma \ref{lem:moddev}). Finally, it shall be more convenient to write $u_n \ll v_n$ for $u_n = o(v_n)$.

\subsection{Motivation from Random Matrix Theory}

Let us explain the incentive for this model. Our goal is to study the spectral measure $\mu$ of the Circular Unitary Ensemble (CUE) of order $n$. The monic orthogonal polynomials $\phi_0,\dots,\phi_{n-1}$ for this measure obey the recurrence relation
\begin{equation} \label{eq:szegorec}
\phi_{k+1}(z) = z \phi_k(z) - \bar{\b_k} z^k \bar{\phi_k}(z),
\end{equation}
see \cite{SimonOPUCOOF} or \cite{Simon1}. In \cite{KillipNenciu}, Killip and Nenciu gave a matrix model for the CUE, and they showed that the coefficients $(\b_k)_{k=0,\dots,n-2}$ are independent and circularly symmetric.

Writing $X_n(z) = z^{-n} \phi_n(z)$, Equation \eqref{eq:szegorec} readily implies that
\[
\frac{X_{n+1}(z) - X_n(z)}{X_n(z)} = - \bar{\b_n} z^{n-1} \frac{\bar{X_n(z)}}{X_n(z)}.
\]
The last term $\bar{X_n(z)}/X_n(z)$ gives rise to the greatest difficulties in studying this problem, so our first step is to merely ignore it. Secondly, we may consider i.i.d. circularly symmetric coefficients $(\b_k)_{k \geq 0}$. We are thus led to the recurrence
\[
\frac{X_{n+1}(z) - X_n(z)}{X_n(z)} = - \bar{\b_n} z^{n-1}.
\]
This essentially means
\[
\ln X_{n+1}(z) - \ln X_n(z) = - \bar{\b_n} z^{n-1}
\]
what is precisely the problem of coupled random walks that we study.

Finally, we want to study fine properties of the spectral measure. According to some seminal results of Jitormiskaya and Last (see \cite{JL}, and \cite{Simon2} in the OPUC case), the asymptotic behavior of $(\phi_n)$ is related to the local Hausdorff dimension of the spectral measure. Studying exceptional behavior for $(\phi_n)$ thus allows to study exceptional points for this measure.

Nonetheless, despite these strong reductions, the problem does not turn out to be trivial, as we shall see.

\section{Techniques}

\subsection{Outline of the article}

We shall first prove the result for a Gaussian random walk, then for the general rotationally symmetric one. In both cases, the proof consists mainly of three steps.

The first one is to give precise first- and second-order estimates for moderate deviations, namely for the probabilities
\[
\P ( |S_n(\th)| > \phi(n) ), \quad \P ( |S_n(\th)| > \phi(n), |S_n(\th')|> \phi(n) ).
\]

We then construct the infinite complete binary tree, and circle some vertices as follows. Fix $q > 1$, and consider the $i$-th vertex at level $n$. Compute $S_{q^n}(i 2^{-n})$, remembering that we shall always write $q^n$ instead of $\lfloor q^n \rfloor$. Then we circle the vertex if
\[
|S_{q^n}(i 2^{-n})| > \phi(q^n).
\]
The estimates from above allow to compute the Hausdorff dimension of the set of rays (i.e. paths from the root to infinity) containing infinitely many circled vertices.

This is close to what we wish to compute, but the tree construction only shows a partial image of the process, since it is sampled at specific times and angles. The last step is thus to fill in the gaps. This is the reason for sampling at time $q^n$: taking $q$ close to 1 allows to control the time variations, whereas $q > 2$ allows to control the angular variations.

\subsection{Chatterjee's invariance principle}

The main simplification for Gaussian random walks obviously relies on the fact that for each $n$, $B_n(\th)$ is a Gaussian random variable, and even more, that $(B_n(\th),B_n(\th'))$ is a complex Gaussian vector, which allows to compute its density function and thus the moderate deviations probabilities easily. This is the main difficulty for general variables, as should be clear from the quite tedious computations in Section \ref{sec:moddev2}. To perform this computation, we compare the variables to Gaussian variables, the main tool being Chatterjee's invariance principle, as introduced in \cite{Chatterjee}.

A multivariate version is given in \cite{SenVirag}, and we will now give another version taking into account complex variables, and, more importantly, sharper, since it goes one order further in the Taylor expansion. In the following statement, we take $\bfX = (X_1,\dots,X_n)$ and $\bfY = (Y_1,\dots,Y_n)$ two vectors of independent complex random variables, each with four moments. We assume that for each $i \in \{1,\dots,n\}$, 
\[
\E(\Re(X_i)^k \Im(X_i)^l) = \E(\Re(Y_i)^k \Im(Y_i)^l)
\]
for every $1 \leq k + l \leq 3$. As expected, we wish $\bfY$ to be Gaussian, so we may tune $\bfX$ so that the two first moments match, but matching the third ones can clearly be done only in some specific cases. In particular, if, as in our case, $X_i$ has a rotationally invariant distribution, then the third moments are zero and this result can be applied.

We shall call a function from $\C^n$ to $X$, with $X = \R^m$ or $\C^m$, $m \geq 1$, $k$ times continuously differentiable if it is $k$ times continuously differentiable as a mapping from $\R^{2 n}$ to $X$. This is an important and necessary weakening of the natural assumption of holomorphy, since the functions we want to consider are typically plateau functions, which are zero for $|z| \leq 1$ and 1 for $|z| \geq 1 + \eps$, which can clearly be made $C^{\infty}$ but not holomorphic.

If $H \, : \, \C^n \to \R$ is a $r$ times differentiable mapping, we may write $H(z_1,\dots,z_n) = H(x_1,y_1,\dots,x_n,y_n)$ and define the partial derivatives $\partial^{k+l} H / \partial x_i^k \partial y_j^l$ for $k + l \leq r$. For $u \in \C^n$ and $z = x + i y \in \C$, we let
\[
D_j^r(H)(u).z = \sum_{k=0}^r \binom{r}{k} x^k y^{r-k} \dn{r}{H}{x_j^k y_j^{r-k}}(u).
\]

Let us now fix $f \, : \, \C^n \to \C^m$, $m \geq 1$, four times continuously differentiable, and $U=f(\bfX)$, $V=f(\bfY)$.

\begin{lemma} \label{lem:invprinciple}
For any $g \, : \, \C^m \to \R$ four times continuously differentiable,
\[
\left | \E(g(U)) - \E(g(V)) \right | \leq \sum_{j=1}^n \E(R_j) + \sum_{j=1}^n \E(T_j)
\]
where
\[
R_j = \frac{1}{24} \sup_{z \in [0,X_j]} \left | D_j^4(g \circ f)(X_1,\dots,X_{j-1},z,Y_{j+1},\dots,Y_n).X_j \right |
\]
and
\[
T_j = \frac{1}{24} \sup_{z \in [0,Y_j]} \left | D_j^4(g \circ f)(X_1,\dots,X_{j-1},z,Y_{j+1},\dots,Y_n).Y_j \right |.
\]
\end{lemma}

The proof is just a perusal of the arguments of \cite{Chatterjee} or \cite{SenVirag}, using the multidimensional Taylor formula, and noting that the assumptions on the moments of the $X_i$'s and the $Y_i$'s are precisely what is needed to cancel out the first, second and third order terms in the Taylor expansion.

\subsection{Tree dimension} \label{sec:treedim}

\subsubsection{Setting}

Let us explain more precisely how we will define our tree and compute Hausdorff dimensions. To begin with, construct a complete infinite binary tree $\cT$, calling $v_i^n$, $i \in \{ 0, \dots, 2^n - 1\}$ its vertices at level $n$. We shall write $|v|$ for the level of $v$ (where the root has level 0), and $u \prec v$ is $u$ is an ancestor of $v$.

Each angle $\th \in [0,1)$ corresponds to a ray (i.e. a path from the root to infinity) $\cR(\th)$ in this tree, by saying that if $\th$ has a proper binary expansion $\th = 0,b_1 b_2 \dots$, then $\cR(\th)$ is the path in the tree starting from the origin and going left (resp. right) at level $i \geq 0$ if $b_{i+1} = 0$ (resp 1). Clearly, $\cR$ is not onto all rays.

Let us reformulate this. Define
\[
A_i^n = [i 2^{-n},(i+1)2^{-n}), \quad i \in \{0, \dots, 2^n - 1 \}
\]
and let $i^n(\th)$ be the $i$ such that $\th \in A_i^n$. Then
\[
\cR(\th) = \{ v^n_{i^n(\th)}, \, n \geq 0 \}.
\]

Let us suppose that we are given a collection of random variables $(Z_v, v \in \cT)$, indexed by the tree, with values in $[0,1]$. One may think that we circle $v$ when $Z_v > 0$. Our interest is the limsup fractal associated to $(Z_v)$, in the terminology of \cite{MP}, defined as the set of angles with infinitely many circled vertices on their path, to wit
\begin{align*}
D & = \{ \th \in [0,1), \, \# \{ v \in \cR(\th), Z_v > 0 \} = \pinf \} \\
& = \{ \th \in [0,1), \, Z_{v^n_{i^n(\th)}} > 0 \; \text{i.o.} \} \\
& = \bigcap_{N \in \N} \bigcup_{n \geq N} \bigcup_{i = 0}^{2^n - 1} A_i^n \unn{Z_{v_i^n} > 0},
\end{align*}
where, for a set $X$, $X.1 = X$ and $X.0 = \emp$. We will finally assume that the law of $Z_v$ is the same for the vertices at the same level, and let
\[
p_n = \P(Z_v > 0), \quad m_n = \E(Z_v), \quad |v| = n.
\]
Note that $p_n \geq m_n \geq \P(Z_v = 1)$. We shall now give two results concerning upper and lower bounds on the Hausdorff dimension of $D$.

\subsubsection{Upper bound}

\begin{lemma} \label{lem:upbound}
Assume that $p_n = O (2^{-n \b})$. Then
\[
\dim D \leq 1 - \b
\]
almost surely.
\end{lemma}

\begin{proof}
This is a well-known result in various guises, see e.g. \cite{MP}. It suffices to notice that for each $N \in \N$,
\[
D \subset \bigcup_{n \geq N} \bigcup_{i = 0}^{2^n - 1} A_i^n \unn{Z_{v_i^n} > 0}.
\]
Hence, the $\g$-Hausdorff content $H^{\g}(D)$ of $D$ verifies
\[
H^{\g}(D) \leq \sum_{n \geq N} \sum_{i=0}^{2^n - 1} |A_i^n|^{\g} \unn{Z_{v_i^n} > 0} = \sum_{n \geq N} \sum_{i=0}^{2^n - 1} 2^{-n \g} \unn{Z_{v_i^n} > 0},
\]
and thus
\[
\E(H^{\g}(D)) \leq \sum_{n \geq N} \sum_{i=0}^{2^n - 1} 2^{-n \g} \P(Z_{v_i^n} > 0) = \sum_{n \geq N} 2^n 2^{-n \g} p_n.
\]
By assumption on $p_n$, this sum is finite whenever $\g > 1 - \b$, and in this case, having $N$ tend to $\pinf$ shows that $\E(H^{\g}(D)) = 0$, so that $H^{\g}(D) = 0$ a.s. Thus
\[
\dim D \leq \g
\]
almost surely. Since this holds for every $\g > 1 - \b$, the result follows.
\end{proof}

\subsubsection{Lower bound}

For the lower bound, we shall use a version of Theorem 10.6 in \cite{MP}, reformulated in our context. For $n \geq 0$ and $|u| \leq n$, define
\[
M_n(u) = \sum_{v \succ u, \, |v| = n } Z_v.
\]

\begin{lemma} \label{lem:lowbound}
Assume that there exist $\zeta(n) \geq 1$ and $0 < \g < 1$ such that
\begin{itemize}
	\item for every $m \leq n$ and $|u| = m$, $\Var(M_n(u)) \leq \zeta(n) \E(M_n(u)) = \zeta(n) m_n 2^{n-m}$,
	\item $2^{n (\g - 1)} \zeta(n) m_n^{-1} \to 0$ as $n \to \pinf$.
\end{itemize}
Then $\dim D \geq \g$ almost surely.
\end{lemma}

This is not exactly Theorem 10.6 in \cite{MP}, since in this reference, vertices are either black or white (i.e. $Z_v \in \{0,1\}$), whereas we allow all shades of gray (i.e. $Z_v \in [0,1]$). However, with our definitions, the result still holds. Checking this requires a careful perusal of the arguments of \cite{MP}. We give below some more details which can be skipped in a first reading. 

\begin{proof}
The idea of the proof of Lemma \ref{lem:lowbound} is to construct a probability measure supported by $D$ (in the sense that $\mu(D^c) = 0$), which has finite $\g$-energy. To this end, one can find an increasing sequence $(\ell_n)$ such that $M_{\ell_k}(u) > 0$ for every $|u| = \ell_{k-1}$. One may then define a probability measure consistently on every dyadic interval by
\begin{itemize}
	\item assigning mass $2^{- \ell_0}$ to each interval $[u,u+2^{- \ell_0})$, $|u|=l_0$;
	\item defining recursively, for $|u| = m$, $\ell_{k-1} < m \leq \ell_k$ and $v$ the ancestor of $u$ at level $\ell_{k-1}$,
\[
\mu ( [u, u + 2^{-m}) ) = \frac{M_{\ell_k}(u) \mu([v,v+2^{- \ell_{k-1}}))}{M_{\ell_k}(v)}.
\]
\end{itemize}
The remaining of the proof in \cite{MP} shows that there is a relevant choice of $(\ell_n)$ such that this measure has finite $\g$-energy, and one can check that this requires no modification but writing $(Z_v)^2 \leq Z_v$ instead of $(Z_v)^2 = Z_v$.

The proof is then over once we check that this measure is supported by $D$. To see this, note that if $\theta \in D^c$, then for some $v \in \cR(\th)$ and every $u \in \cR(\th)$ with $u \succ v$, $Z_u = 0$. Then, for $k$ large enough $\ell_{k-1} \geq |v|$, and the very construction of the measure then implies that, for $u \in \cR(\th)$ at level $\ell_k$,
\[
\mu ( [u, u + 2^{-\ell_k}) ) = \frac{M_{\ell_k}(u) \mu([v,v+2^{- \ell_{k-1}}))}{M_{\ell_k}(v)} = \frac{Z_u \mu([v,v+2^{- \ell_{k-1}}))}{M_{\ell_k}(v)} = 0.
\]
Since $\th \in [u,u+2^{-\ell_k})$, this tells that there exists $\eps > 0$ such that $\mu([\th,\th+\eps)) = 0$, so that indeed $\mu(D^c) = 0$.
\end{proof}

\subsection{Bernstein's inequality}

A last tool that we are going to use is the classical Bernstein inequality \cite{BernsteinIneq}, which we recall here for the reader's convenience.

\begin{lemma}
Let $X_1,\dots,X_n$ be independent centered random variables. Assume that there is a $M$ such that $|X_i| \leq M$ for all $i$. Then, for all $t > 0$,
\[
\P \left ( \sum_{i=1}^n X_i > t \right ) \leq \exp \left ( - \frac{t^2/2}{\sum_{i=1}^n \E(X_i^2) + M t / 3}\right ).
\]
\end{lemma}

\section{Random walk with Gaussian increments}

\subsection{Moderate deviations} \label{sec:moddev}

Computing the Hausdorff dimension relies on first and second moment estimates which we now give.

\begin{lemma} \label{lem:moddev}
The following estimates hold.
\begin{enumerate}
\item For every $\th \in [0,1)$,
\[
\P \left ( |B_n(\th)| > \phi(n) \right ) = \frac{1}{n^{\a}}.
\]
\item Fix a sequence $\log n / n \ll \eps_n \ll 1$. Then,
\[
\P \left ( |B_n| > \phi(n), |B_n(\th)|> \phi(n) \right ) = \frac{1}{n^{2 \a}} \left ( 1 + O \left ( \frac{\log n}{n \th} \right ) \right )
\]
uniformly for $\th \in [\eps_n,1)$, and in particular
\[
\left | \P \left ( |B_n| > \phi(n), |B_n(\th)|> \phi(n) \right ) - \P \left ( |B_n| > \phi(n) \right ) \P \left ( |B_n(\th)|> \phi(n) \right ) \right | = \frac{1}{n^{2 \a}} O \left ( \frac{\log n}{n \th} \right )
\]
still uniformly in $\th \in [\eps_n,1)$.
\item Fix a sequence $\log n / n \ll \eps_n \ll 1$ and a bounded measurable function $g \, : \, \R^+ \to \R^+$, with support in $[1, \pinf)$. Then,
\[
\left | \E \left ( g( | \tB_n | ) g ( | \tB_n(\th) | ) \right ) - \E \left ( g ( | \tB_n | ) \right ) \E \left ( g ( | \tB_n(\th) | ) \right ) \right | = \frac{1}{n^{2 \a}} O \left ( \frac{\log n}{n \th} \right )
\]
uniformly in $\th \in [\eps_n,1)$.
\end{enumerate}
\end{lemma}

These results essentially mean that the events $\{ |B_n| > \phi(n) \}$ and $\{ |B_n(\th)| > \phi(n) \}$ are independent if $\th \gg \log n / n$. Lemma \ref{lem:angdev} will show that, on the other hand, these events are essentially identical when $\th \ll n^{-1-\b}$, for $\b > 0$. There is thus a small window of uncertainty, but too small to be of any harm.

\begin{proof}
\begin{itemize}[wide]
\item We already mentioned that the first equality stems from the fact that $|B_n|^2/n$ has a $\chi^2$ distribution with two degrees of freedom. Now, note that
\[
\frac{1}{\sqrt{2n}} (B_n,B_n(\th))
\]
is a complex Gaussian vector with mean $(0,0)$, zero relation matrix, and covariance matrix
\[
\begin{pmatrix}
1 & \bar{D_n}(\th) \\
D_n(\th) & 1
\end{pmatrix}
\]
where
\[
D_n(\theta) = \frac1n \sum_{j=1}^n e^{2 i \pi j \th} = \frac1n e^{i \pi (n-1)\th} \frac{\sin \pi n \th}{\sin \pi \th}
\]
is almost the Dirichlet kernel. Hence, $(B_n, B_n(\th))$ has density in $\C^2$ given by
\[
f(z_1,z_2) = \frac{1}{\pi^2(1- |D_n(\th)|^2)} \exp - \frac{1}{(1- |D_n(\th)|^2)} \left ( |z_1|^2 + |z_2|^2 - 2 \Re(D_n(\th) z_1 \bar{z_2}) \right ).
\]
Note that
\[
\left | 2 \Re(D_n(\th) z_1 \bar{z_2}) \right | \leq |D_n(\th)| ( |z_1|^2 + |z_2|^2 )
\]
so that, writing $R_n = \phi(n) / \sqrt{2n} = \sqrt{\a \log n}$,
\[
\begin{split}
\P(|B_n| > \phi(n), & \, |B_n(\th)| > \phi(n)) = \int_{|z_1| > R_n} \int_{|z_2| > R_n} f(z_1,z_2) \dz_1 \dz_2 \\
& \leq \frac{1}{\pi^2(1- |D_n(\th)|^2)} \int_{|z_1| > R_n} \int_{|z_2| > R_n} \exp - \frac{1}{1 + |D_n(\th)|}( |z_1|^2 + |z_2|^2) \dz_1 \dz_2 \\
& = \frac{1 + |D_n(\th)|}{1- |D_n(\th)|} \exp - \frac{1}{1 + |D_n(\th)|} 2 \a \log n.
\end{split}
\]
Now, uniformly for $\th \in [\eps_n,1]$ (as in all the following computations),
\[
|D_n(\th)| = O \left ( \frac{1}{n \th} \right )
\]
so $|D_n(\th)| \log n$ tends to 0 as $n \to \pinf$, by assumption on $(\eps_n)$. Then, by straightforward computations
\begin{align*}
\frac{1 + |D_n(\th)|}{1- |D_n(\th)|} \exp - \frac{1}{1 + |D_n(\th)|} 2 \a \log n & = \frac{1}{n^{2 \a}} \left ( 1 + O \left ( |D_n(\th)| \log n \right ) \right ) \\
& = \frac{1}{n^{2 \a}} \left ( 1 + O \left ( \frac{\log n}{n \th} \right ) \right ).
\end{align*}
The lower bound is obtained similarly and provides the result.
\item For the last part, write in the same way
\begin{align*}
& \E \left ( g ( | \tB_n | ) g ( |\tB_n(\th)| ) \right ) = \int_{|z_1| > R_n} \int_{|z_2| > R_n} g(|z_1|) g(|z_2|) f(z_1,z_2) \dz_1 \dz_2 \\
& \leq \frac{1}{\pi^2(1- |D_n(\th)|^2)} \int_{|z_1| > R_n} \int_{|z_2| > R_n} g(|z_1|) g(|z_2|) \exp - \frac{|z_1|^2 + |z_2|^2}{1 + |D_n(\th)|} \dz_1 \dz_2 \\
& = \left ( \frac{2}{\sqrt{1 - |D_n(\th)|^2}} \int_{r >  R_n} r g(r) \exp - \frac{1}{1 + |D_n(\th)|} r^2 \dr \right )^2.
\end{align*}
Using the boundedness of $g$ and the fact that $|D_n(\th)| \log n \to 0$, it is easy to check that
\[
\frac{1}{\sqrt{1 - |D_n(\th)|^2}} \int_{r >  R_n} r g(r) \exp - \frac{1}{1 + |D_n(\th)|} r^2 \dr = \int_{r >  R_n} r g(r) \exp - r^2 \dr + O \left ( \frac{\log n}{n \th} \right ).
\]
On the other hand
\[
\E \left ( g(|\tB_n(\th)|) \right ) = \frac{1}{\pi} \int_{|z| > R_n} g(z) \exp - |z^2| \dz = 2 \int_{r > R_n} g(r) \exp - r^2 \dr
\]
so
\[
\E \left ( g(|\tB_n|) g(|\tB_n(\th)|) \right ) - \E \left ( g(|\tB_n|) \right ) \E \left ( g(|\tB_n(\th)|) \right ) \leq  O \left ( \frac{\log n}{n \th} \right ). 
\]
Similar computations provide the lower bound, and the result follows thereof.
\end{itemize}
\end{proof}

\subsection{Angular deviations}

\begin{lemma} \label{lem:angdev}
Fix $\beta > 0$, a sequence $\eps_n \ll n^{-(1+\b)}$ and $\eta > 0$. Then there exists a sequence $K_n$, diverging to $\pinf$, depending only on $\eps_n$ and $\eta$, such that
\[
\P \left ( \sup_{ | \th - \th' | < \eps_n} |B_n(\th) - B_n(\th')| > \eta \phi(n) \right ) = O \left ( n^{- K_n } \right ).
\]
\end{lemma}

\begin{proof}
Clearly, by rotational invariance, it is enough to prove it for $\th' = 0$. Fix $k \in \N$ such that $1/k < \beta$, and write, with Taylor's formula, that for $\th \in [0,\eps_n]$
\[
B_n(\th) - B_n = \sum_{j=1}^{k - 1} \frac{B_n^{(j)}(0)}{j!} \th^j + R_k(\th)
\]
where
\[
|R_k(\th)| \leq \th^k \frac{1}{k!} \sup_{\th \in [0,\eps_n]} |B_n^{(k)}(\th)|.
\]
Now,
\[
B_n^{(j)}(\th) = \sum_{r=1}^{n} G_r (2 i \pi r)^j e^{2 i \pi r \th}
\]
so
\[
B_n^{(j)}(0) = \sum_{r=1}^n G_r (2 i \pi r)^j \eqlaw \cN_{\C} (0,\sum_{r=1}^n (2 \pi r)^{2j}) \eqlaw \sqrt{ \sum_{r=1}^n (2 \pi r)^{2j} } \, \cN_{\C} (0,1)
\]
and
\[
|R_k(\th)| \leq \th^k \frac{1}{k!} (2 \pi n)^k \sum_{r=1}^n |G_r|.
\]

Gathering the pieces and writing $\sqrt{ \sum_{r=1}^n (2 \pi r)^{2j} } \leq \sqrt{2 \pi} \sqrt{n} n^j$, we may compute, for $C$ some large enough constant depending only on $k$,
\begin{align*}
\P \left ( \vphantom{\sup_{\th \in [0,\eps_n]}} \right. & \left. \sup_{\th \in [0,\eps_n]} |B_n(\th) - B_n| > \eta \phi(n) \right ) \\
& \leq \sum_{j=1}^{k-1} \P \left ( \frac{B_n^{(j)}(0)}{j!} \sup_{\th \in [0,\eps_n]} \th^j  \geq \frac1k \eta \phi(n) \right ) + \P \left ( \sup_{\th \in [0,\eps_n]} |R_k(\th)| \geq \frac1k \eta \phi(n) \right ) \\
& \leq  \sum_{j=1}^{k-1} \P \left ( C \sqrt{n} n^j \eps_n^j |\cN_{\C}(0,1)| \geq \frac1k \eta \phi(n) \right ) + \sum_{r=1}^n \P \left ( C n^k \eps_n^k |G_r| \geq \frac{1}{k n} \eta \phi(n) \right ) \\
& \leq  \sum_{j=1}^{k-1} \exp \left ( - \a \log n \left( \frac{\eta}{k C} \right )^2 \frac{1}{(n \eps_n)^{2j}} \right )+ n \exp \left ( - \a \log n \left ( \frac{\eta}{k C} \right )^2 \frac{1}{(n^{k+1} \eps_n^k)^2} \right ).
\end{align*}
But our choice of $k$ ensures that $n^{k+1} \eps_n^k$ tends to 0, whence the result follows immediately.
\end{proof}

\subsection{Tree-dimension}

We shall now use these results in a case which is relevant to the study of $\cD'_{\a}$. Fix $q > 1$. For $v = v_i^n \in \cT$, define $Z_v(\om) = 1$ if
\[
\om \in E_i^n := \{ |B_{q^n}(i 2^{-n})| > \phi(q^n) \},
\]
and $Z_v = 0$ otherwise. Define the corresponding limsup fractal
\[
D'_{\a} = \{ \th \in [0,1), \, |B_{q^n}(i^n(\th) 2^{-n})| > \phi(q^n) \; \text{i.o.} \}.
\]
We shall prove the following.

\begin{prop} \label{prop:treedim}
Almost surely
\[
(1 - \a) \log_2 q \leq \dim D'_{\a} \leq (1 - \a) \vee 0
\]
for $q \in (1,2)$, and
\[
\dim D'_{\a} =  (1 - \a \log_2 q) \vee 0
\]
for $q > 2$.
\end{prop}

\begin{remark}
Our main interests are actually the lower bound for $q > 2$, and the upper bound for $q \in (1,2)$.
\end{remark}

\begin{proof}
\begin{enumerate}[wide]
\item By Lemma \ref{lem:moddev}, $\P(Z_v > 0) = q^{- n \a}$, so the upper bound for $q > 2$ is a direct corollary of Lemma \ref{lem:upbound}.
\item To obtain the (better) upper bound for $1 < q < 2$, let us mimic the proof of Lemma \ref{lem:upbound}, by providing a better cover of $D'_{\a}$ in this particular case. Lemma \ref{lem:moddev} ensures that if $B_{q^n}(\th)$ is large, then it should also hold for any angle $\th'$ with $|\th' - \th| \leq q^{-n(1+\b)}$, which allows us to get a better cover of $D'_{\a}$.

Let us formalize this idea. Take $\b > 0$, $1 > \eta > 0$, $\eps_n = q^{-n(1+\b)}$ and
\[
B_i^n = [i \eps_n, (i+1) \eps_n), \quad i \in \{ 0, \dots, \lfloor \eps_n^{-1} \rfloor \}.
\]
Write
\[
X_i^n =
\begin{cases}
B_i^n & \text{if $|B_{q^n}(i 2^{-n})| > (1 - \eta) \phi(q^n)$,} \\
\emp & \text{otherwise.}
\end{cases}
\]
By Lemma \ref{lem:angdev}
\begin{align*}
\P \left ( \exists i \in \{ 0, \dots, \lfloor \eps_n^{-1} \rfloor \} \right. & \left. \; \sup_{\th \in B_i^n} |B_{q^n}(\th) - B_{q^n}(i q^{-n})| < \eta \phi(q^n) \right ) \\
& \leq \lfloor (\eps_n^{-1} \rfloor + 1) q^{-n K_n} \leq (q^{ - n(1+\b) } + 1) q^{ - n K_n}
\end{align*}
for some sequence $K_n$ diverging to $\pinf$, and thus Borel-Cantelli's lemma ensures that almost surely, for large enough $n$,
\[
\sup_{i \in \{ 0, \dots, \lfloor \eps_n^{-1} \rfloor \}} \sup_{\th \in B_i^n} |B_{q^n}(\th) - B_{q^n}(i q^{-n})| < \eta \phi(q^n).
\]
It is then easy to check that, for every $N \in \N$,
\[
D'_{\a} \subset \bigcup_{n \geq N} \bigcup_{i=0}^{\eps_n^{-1}} X_i^n.
\]
Now, by Lemma \ref{lem:moddev},
\[
\P \left (  X_i^n \neq \emp \right ) = q^{-n \a (1 - \eta)^2}
\]
so concluding as in Lemma \ref{lem:upbound}, one readily obtains that, almost surely,
\[
\dim D'_{\a} \leq  1 - \a \frac{(1-\eta)^2}{1+\b}.
\]
Since this is true for any $\eta > 0$ and $\beta > 0$, the result follows.

\item To obtain the lower bound for $q > 2$, define as in Lemma \ref{lem:lowbound}, for $n \geq 0$ and $|u| \leq n$,
\[
M_n(u) = \sum_{u \prec v, \, |v| = n } Z_v.
\]
Then, with obvious notation,
\begin{align*}
\Var(M_n(u)) & = \sum_v \E(Z_v^2) - \sum_v \E(Z_v)^2 + \sum_{v \neq v'} \E(Z_v Z_{v'}) - \sum_{v \neq v'} \E(Z_v) \E(Z_{v'}) \\
& \leq \sum_v \E(Z_v) + \sum_{v \neq v'} \left | \E(Z_v Z_{v'}) - \E(Z_v) \E(Z_{v'}) \right | \\
& = \sum_{k = (j-1) 2^{n-m} + 1}^{j 2^{n-m}} \P(E_k^n) +  \sum_{\substack{k, \, l = (j-1) 2^{n-m} + 1 \\ k \neq l}}^{j 2^{n-m}} \left | \P(E_k^n \cap E_l^n) - \P(E_k^n) \P(E_l^n) \right | \\
& = \sum_{k = 1}^{2^{n-m}} \P(E_k^n) +  \sum_{\substack{k, \, l = 1 \\ k \neq l}}^{2^{n-m}} \left | \P(E_k^n \cap E_l^n) - \P(E_k^n) \P(E_l^n) \right | \\
& = 2^{n-m} \P(E_1^n) + 2^{n-m} \sum_{l=2}^{2^{n-m}}  \left | \P(E_1^n \cap E_l^n) - \P(E_1^n) \P(E_l^n) \right |
\end{align*}
where the last equalities stem from the rotational symmetry.

By Lemma \ref{lem:moddev},
\[
\P(E_1^n) = q^{-n \a}.
\]
Now, take $\eps_n = 2^{-n}$, so $n q^{-n} \ll \eps_n \ll 1$. Thus, Lemma \ref{lem:moddev} applies, and provides
\[
\left | \P(E_1^n \cap \E_l^n) - \P(E_1^n) \P(E_l^n) \right | \leq C \frac{1}{q^{2 \a n}} \frac{n}{q^n 2^{-n} (l-1)}
\]
for some constant $C$, all $ n \geq 1$ and all $l \in \{ 2, \dots , 2^n \}$. Then, it is easy to compute
\begin{align*}
\Var(M_n(u)) & \leq 2^{n-m} \left ( q^{-n \a} +  C n q^{- 2 \a n} q^{-n} 2^n \sum_{l=2}^{2^{n-m}} \frac{1}{l-1} \right ) \\
& \leq 2^{n-m} q^{-n \a} \left ( 1 +  C n q^{- \a n} q^{-n} 2^n \times 2 \log(2^n) \right ) \\
& = 2^{n-m} q^{-n \a} O(1)
\end{align*}
since $q > 2$. We may then pick $\zeta(n) = O (1)$ in Lemma \ref{lem:lowbound}, which readily implies the result.

\item The lower bound for $1 < q < 2$ is obtained similarly, but then, one must take into account in the computation of $\Var(M_n(u))$ the contribution of every angle $\th$ less than $q^{-n}$. For these angles, $B_n(\th)$ is very close to $B_n$, and all we can obtain is
\[
\left | \P(E_1^n \cap \E_l^n) - \P(E_1^n) \P(E_l^n) \right | = O (q^{-n \a}).
\]
One thus have to take $\zeta(n) = q^{-n} 2^n + 1 + q^{-n \a} q^{-n} 2^n$, thus providing a lower bound $(1 - \a) \log_2 q$. We do not dwell on the details since, once again, this part of the result will not be used later.
\end{enumerate}
\end{proof}

\begin{remark}
\begin{itemize}
\item The result still holds with $q=2$, but it is unnecessary to us and would make the proof a bit more complicated.
\item The bound
\[
\left | \P(E_1^n \cap \E_l^n) - \P(E_1^n) \P(E_l^n) \right | = O (q^{-n \a}).
\]
for $1 < q < 2$ is essentially optimal, up to a factor $q^{-n \eps_n}$, where $\eps_n \to 0$, which does not improve the computations. Hence, this is the best which can be obtained thanks to Lemma \ref{lem:lowbound}.
\end{itemize}
\end{remark}

\subsection{Exceptional angles}

Recall that our main interest is to consider the ``true'' set of exceptional angles, i.e. those angles $\th$ such that $|B_n(\th)|$ is exceptionally large infinitely often. In formulas
\[
\cD'_{\a} = \{ \th \in [0,1), \; |B_n(\th)| > \phi(n) \; \text{i.o.} \}.
\]
We wish to prove \eqref{eq:dimBn}, to wit that $\dim \cD'_{\a} = (1 - \a) \vee 0$ a.s.

\begin{proof}[Proof of \eqref{eq:dimBn}]
We proved this equality for the tree-dimension, which is obtained by sampling our process at specific angles and times. The idea is then to prove than in between these times and angles, things cannot go too bad, i.e. the process does not vary much. More specifically, for the lower bound, we only need to control the angular variations, since we already know that $B_n$ is large i.o.\ on a large set of rays. For the upper bound, we need to control both the angular and time variations, to see that if $B_{q^n}(i2^{-n})$ is not too large, then it is also the case for every angle in $[i2^{-n},(i+1)2^{-n})$ and time $q^n + 1, \dots, q^{n+1}$ that the tree does not see.

In our notation, we will consider the set
\[
D'_{(1 + \eta)^2 \a} := \{ \th \in [0,1), \; |B_{q^n}(i^n(\th) 2^{-n})| > (1 + \eta) \phi(q^n)\, \text{i.o.} \}
\]
for $\eta > - 1$.

\begin{enumerate}[wide]
\item Let us start with the lower bound. Let us fix $\eta > 0$, $q > 2$, and define
\[
\underline{\cD}'_{\eta} = \{ \th \in [0,1), \; \sup_{x \in A^n_{i^n(\th)}} |B_{q^n}(x) - B_{q^n}(i^n(\th) 2^{-n})| > \eta \phi(q^n) \; \text{i.o.} \}.
\]
The set $\underline{\cD}'_{\eta}$ is precisely the limsup fractal associated to $(Z_v)$, if we let $Z_{v_i^n} = 1$ when
\[
\sup_{\th \in A_i^n} |B_{q^n}(\th) - B_{q^n}(i 2^{-n})| > \eta \phi(q^n)
\]
and 0 otherwise. But
\begin{align*}
\P \left (\sup_{x \in A_i^n} |B_{q^n}(x) - B_{q^n}(i 2^{-n})| > \eta \phi(q^n) \right )
\end{align*}
decays faster than any power of $q^n$ by Lemma \ref{lem:angdev} so Lemma \ref{lem:upbound} ensures that
\[
\dim \underline{\cD}'_{\eta} = 0.
\]
On the other hand
\[
\dim D'_{(1+\eta)^2 \a} = 1 - \a (1 + \eta)^2 \log_2 q
\]
by Proposition \ref{prop:treedim}. Finally, it is clear that
\[
D'_{(1+\eta)^2 \a} \bsl \underline{\cD}'_{\eta} \subset \cD'_{\a}
\]
so
\[
\dim \cD'_{\a} \geq 1 - \a (1 + \eta)^2 \log_2 q.
\]
Since this is valid for any $\eta > 0$ and $q > 2$, the result follows.

\item To get the lower bound, fix $0 < \eta <  1$, $q = 1 + \eta^2 \a / (4 (1 + \a))$, and define similarly
\[
\bar{\cD}'_{\eta} = \{ \th \in [0,1), \; \sup_{\substack{x \in A^n_{i^n(\th)} \\ r \in \{ q^n + 1 ,\dots, q^{n+1} \} }} |B_r(x) - B_{q^n}(i^n(\th) 2^{-n}))| > \eta \phi(q^n) \; \text{i.o.} \},
\]
which is the limsup fractal associated to $(Z_v)$, if we let $Z_{v_i^n} = 1$ when
\[
\sup_{\substack{x \in A^n_{i^n(\th)} \\ r \in \{ q^n + 1 ,\dots, q^{n+1} \} }} |B_r(x) - B_{q^n}(i^n(\th) 2^{-n}))| > \eta \phi(q^n).
\]
But we can compute
\begin{align*}
\P & \left ( \sup_{\substack{x \in A_i^n \\ r \in \{ q^n + 1 ,\dots, q^{n+1} \} }} |B_r(x) - B_{q^n}(i 2^{-n})| > \eta \phi(q^n) \right ) \\
& \leq \sum_{r=q^n + 1}^{q^{n+1}} \P \left ( \sup_{x \in A_i^n} |B_r(x) - B_{q^n}(i2^{-n})|  > \eta \phi(q^n) \right ) \\
& \leq \sum_{r=q^n + 1}^{q^{n+1}} \left ( \P \left ( \sup_{x \in A_i^n} |B_r(x) - B_r(i2^{-n})|  > \frac{\eta \phi(q^n)}{2} \right ) + P \left ( |B_r(i2^{-n}) - B_{q^n}(i2^{-n})| > \frac{\eta \phi(q^n)}{2} \right ) \right ).
\end{align*}
By Lemma \ref{lem:angdev}, there exists $K_n$ diverging to $\pinf$ (which we may assume increasing) and a constant $C > 0$ such that, whenever $q^n < r \leq q^{n+1}$,
\[
\P \left ( \sup_{x \in A_i^n} |B_r(x) - B_r(i2^{-n})|  > \frac{\eta \phi(q^n)}{2} \right ) \leq C r^{- K_r} \leq C q^{-n K_{q^n}}.
\]
On the other hand, still for $q^n < r \leq q^{n+1}$,
\[
B_r(i2^{-n}) - B_{q^n}(i2^{-n}) \eqlaw \cN_{\C} (0, r - q^n + 1) = \sqrt{r - q^n + 1} \cN_{\C}(0,1),
\]
so
\begin{align*}
P \left ( |B_r(i2^{-n}) - B_{q^n}(i2^{-n})| > \frac{\eta \phi(q^n)}{2} \right ) & \leq \exp - \frac{\a \eta^2 q^n \log(q^n)}{4(r - q^n + 1)} \\
& \leq \exp - \frac{\a \eta^2 q^n \log(q^n)}{4 q^n(q-1)} = q^{-n \a \eta^2/(4(q-1))} = q^{-n} q^{- n \a}
\end{align*}
by our very choice of $q$. Gathering the pieces, we thus obtain that
\[
\P \left ( \sup_{\substack{x \in A_i^n \\ r \in \{ q^n + 1 ,\dots, q^{n+1} \} }} |B_r(x) - B_{q^n}(i 2^{-n})| > \eta \phi(q^n) \right ) = O \left ( q^{- n \a} \right )
\]
so by Lemma \ref{lem:upbound}, this shows that
\[
\dim \bar{\cD}'_{\eta} \leq 1 - \a.
\]
Moreover, by Proposition \ref{prop:treedim},
\[
\dim D'_{(1- \eta)^2 \a} = 1 - \a (1 - \eta)^2 > 1 - \a
\]
To conclude, note that
\[
\cD'_{\a} \subset \cD'_{(1 - \eta)^2 \a} \cup \bar{\cD}'_{\eta}
\]
so that
\[
\dim \cD'_{\a} \leq 1 - \a (1 - \eta)^2.
\]
Since this is valid for any $\eta > 0$, the result follows.
\end{enumerate}
\end{proof}

\section{Random walk with general increments}

\subsection{Result}

Now, as mentioned in the introduction, we are interested in the random walk with rotationally symmetric increments
\[
S_n(\th) = U_1 e^{2 i \pi \th} + \dots + U_n e^{2 i \pi n \th},
\]
where the $U_i$ have an exponential moment. We shall prove Theorem \ref{th:dimSn}, the most general form of \eqref{eq:dimBn}. Recall that we define
\[
\cD = \{\th \in [0,1), \, |S_n(\th)| > \s \phi(n) \; \text{i.o.} \}.
\]
Following the steps of the proof of \eqref{eq:dimBn}, we will prove that, almost surely,
\[
\dim \cD = (1 - \a) \vee 0.
\]
The strategy of proof is the same as for \eqref{eq:dimBn}. The major difference is that computing moderate deviations is more challenging, the remaining of the proof being essentially the same.

\subsection{Moderate deviations} \label{sec:moddev2}

\subsubsection{First-order comparison} \label{sec:moddevSn1}

Thanks to Lemma \ref{lem:invprinciple}, let us give an estimation of $\P(|S_n| > \phi(n))$, and $\E(g(|S_n|))$, where $g$ is a smooth approximation of the indicator function $\unn{\cdot \, > \phi(n)}$. This provides the main ideas in a quite easy setting, and we shall give less details in the next part, where we estimate $\E(g(|S_n|) g(|S_n(\th)|))$.

In the following, fix $p$ a four times continuously differentiable plateau function from $\C$ to $\R$, which is zero on $\{ |z| \leq 1 \}$, $1$ on $\{ |z| \geq 2 \}$, positive on $\{ 1 < |z| < 2 \}$ and nondecreasing in $|z|$. Consider the rescaling $p_{m,\eps}(z) = p(1+(z-m)/\eps)$, which is zero on $\{ |z| \leq m \}$, $1$ on $\{ |z| \geq m + \eps \}$. To simplify notations, we let $g = p_{m,\eps}$ for some $m \leq 1$, $\eps > 0$.

\begin{lemma} \label{lem:moddevSn1}
The estimate
\[
\left | \E \left (  g \left ( |\tS_n| \right ) \right ) - \E \left (  g \left ( |\tB_n| \right ) \right ) \right | = \frac{1}{n^{\a m^2}} O \left( \frac1n \right )
\]
and
\[
\P \left ( |S_n| > \s \phi(n) \right ) = \frac{1}{n^{\a}} \left ( 1 + O \left ( \frac{\log n}{n^{1/5}} \right ) \right )
\]
hold uniformly in $n$.
\end{lemma}

\begin{proof}
\begin{enumerate}[fullwidth]
\item Let us consider $\bfX = \s^{-1} (U_1,\dots,U_n)$ and $\bfY = (G_1,\dots,G_n)$, which we may assume in this proof to be independent. In all the following, $C$ is a constant, which may change from line to line, but only depends on the law of $X_1$. Let
\[
f(z_1,\dots,z_n) = \frac{1}{\phi(n)} (z_1 + \dots + z_n).
\]
Note the rescaling of $\bfX$ so that
\[
\E (\Re(\s^{-1} U_j)^2) = \E (\Im(\s^{-1} U_j)^2) = \E (\Re(G_j)^2) = \E (\Im(G_j)^2) = 1.
\]
Clearly, the other first, second and third moments are all zero for both variables, and we are thus able to use Lemma \ref{lem:invprinciple}.

\item Let us bound $R_j$, in the notation of Lemma \ref{lem:invprinciple}, since bounding $T_j$ is done in a similar manner. To simplify notation, define, for $z \in \C$,
\[
[\bfX,\bfY,z]_j = X_1 + \dots + X_{j-1} + z + Y_{j+1} + \dots + Y_n, \quad [\bfX,\bfY]_j = [\bfX,\bfY,0]_j.
\]
Let $H = g \circ f$. Note first that,
\[
\sup_{z \in \C} \left | \ddn{k+l}{g}{x^k}{y^l}(z) \right | \leq \frac{C}{\eps^{k+l}} \unn{\frac{1}{\phi(n)}|z_1+\dots+z_n| > m}
\]
for every $k+l \leq 4$, since $g$ is 0 on $\{ |z| \leq m \}$. Then, for $k+l=4$,
\[
\ddn{4}{H}{x_j^k}{y_j^l}(z_1,\dots,z_n) = \frac{1}{\phi(n)^4} \ddn{4}{g}{x_j^k}{y_j^l} \left ( \frac{1}{\phi(n)} (z_1,\dots,z_n) \right )
 \leq  \frac{C}{(\eps \phi(n))^4} \unn{\frac{1}{\phi(n)}|z_1+\dots+z_n| > m}.
\]
Then, consider a sequence $1 \ll K_n \ll \phi(n)$, to be fixed later. On $\{ |X_j| \leq K_n \}$, we have
\begin{align*}
\sup_{z \in [0,X_j]} \left | D_j^4(g \circ f)([\bfX,\bfY,z]_j).X_j \right | & \leq \frac{C}{(\eps \phi(n))^4} \sup_{z \in [0,X_j]} |X_j|^4 \unn{|[\bfX,\bfY,z]_j| > m \phi(n)} \\
& \leq \frac{C}{(\eps \phi(n))^4} |X_j|^4 \unn{|[\bfX,\bfY]_j| > m \phi(n) - K_n}.
\end{align*}
Note that $X_j$ is independent from $[\bfX,\bfY]_j$, so finally
\[
\E(R_j) \leq \P(|X_j| > K_n) +  \frac{C}{(\eps \phi(n))^4} \P ( |[\bfX,\bfY]_j| > m \phi(n) - K_n).
\]
The first term is easily dealt with thanks to Markov's inequality, which provides
\[
\P(|X_j| > K_n) \leq C e^{- \k K_n},
\]
where we recall that $\k$ is some constant such that $\E(\exp \k |X_j|) < \pinf$.

\item To bound the second term, we shall use Bernstein's inequality. First, note that for $z \in \C$, $d_n \in \N$,
\[
|z| > 1 \Rightarrow \exists k \in \{1,\dots,d_n\} \; z \cdot e^{2 i k \pi / d_n} \geq \cos \frac{\pi}{d_n}
\]
where $\cdot$ is the scalar product when we see complex numbers as vectors in $\R^2$. So now, if we write
\[
\psi(n) = \left ( m \phi(n) - K_n \right ) \cos \frac{\pi}{d_n} ,
\]
then
\begin{align*}
\P (|[\bfX,\bfY]_j| >  m \phi(n) - K_n) & \leq \sum_{k=1}^{d_n} \P ( [\bfX,\bfY]_j \cdot e^{2 i k \pi / d_n} > \psi(n)) \\
& = d_n \, \P ( \Re([\bfX,\bfY]_j) > \psi(n) ) \\
& = d_n \, \P(C_1 + \dots + C_{j-1} + N_{j+1} + \dots + G_n > \psi(n)),
\end{align*}
where we used the rotational invariance and wrote $X_k = C_k + i C_k'$ and $G_k = N_k + i N_k'$. Now, after truncating the variables at level $K_n$, we may use Bernstein inequality, to get
\begin{align*}
\P(C_1 + \dots +{} & C_{j-1} + N_{j+1} + \dots + N_n > \psi(n)) \\
& \leq \P(C_1 + \dots + C_{j-1} + N_{j+1} \unn{|N_{j+1}| > K_n} + \dots + N_n \unn{|N_n| > K_n} > \psi(n)) \\
& \quad + \P (\exists k \in \{1,\dots,j-1\} \; |C_k| > K_n) + \P (\exists k \in \{j+1,\dots,n\} \; |N_k| > K_n) \\
& \leq \exp - \frac{\psi(n)^2/2}{n + \psi(n) K_n /3} + (j-1) \P (|C_1| > K_n) + (n-j) \P (|N_1| > K_n) \\
& \leq \exp - \frac{\psi(n)^2/2}{n + \psi(n) K_n /3} + C n e^{- \k K_n}.
\end{align*}
Take now $d_n = \log n$ and $K_n = \log^2 n$. Then $\psi(n)^2 = 2 m^2 \a n \log n + O ( n^{3/4} )$, $n + \psi(n) K_n /3 = n (1 + O(n^{-3/4}))$, and thus
\[
\begin{split}
\P(C_1 + \dots + C_{j-1} + N_{j+1} + \dots +{} & N_n > \psi(n)) \\
& \leq \exp ( - \a m^2 \log n + O ( n^{-1/4} )) + C n e^{- \k K_n} \\
& \leq C \frac{1}{n^{\a m^2}}.
\end{split}
\]
Gathering the pieces, we get
\[
\E(T_j) \leq C \left ( e^{- \k K_n} + \frac{1}{(\eps \phi(n))^4} \frac{1}{n^{\a m^2}} \right )
\]
and finally, Lemma \ref{lem:invprinciple} provides
\[
| \E ( f \circ g (\bfX)) - \E (f \circ g (\bfY)) | \leq C \left ( n e^{- \k K_n} + \frac{n}{(\eps \phi(n))^4} \frac{1}{n^{\a m^2}} \right ).
\]
This gives the first part of the result.

\item To get the second part, let us take, in the same notation, $m = 1$, consider $\eps := \eps_n$ has a function of $n$, and assume that $\eps_n \log_n \to 0$. Then
\begin{align*}
& \P ( |\s (U_1 + \dots + U_n)| > \phi(n)) \\
& \geq \E(f \circ g (\bfX)) \\
& \geq \E(f \circ g (\bfY)) - \left | \E(f \circ g (\bfX)) - \E(f \circ g (\bfY)) \right | \\
& \geq \P ( |G_1 + \dots + G_n | > (1 + \eps_n) \phi(n)) - \left | \E(f \circ g (\bfX)) - \E(f \circ g (\bfY)) \right | \\
& \geq \P ( |G_1 + \dots + G_n | > (1 + \eps_n) \phi(n)) + \frac{1}{\eps_n^4} \frac{1}{n^{\a}} O \left( \frac1n \right ) + O \left( n e^{- \k K_n} \right ) \\
& = \frac{1}{n^{\a (1 + \eps_n)^2}} + \frac{1}{\eps_n^4} \frac{1}{n^{\a}} O \left( \frac1n \right ) + O \left( n e^{- \k K_n} \right ) \\
& = \frac{1}{n^{\a}} \left ( 1 + \eps_n O ( \log n ) + \frac{1}{\eps_n^4} O \left( \frac1n \right ) \right ) + O \left( n e^{- \k K_n} \right ).
\end{align*}
This can be roughly optimized by taking $\eps_n = n^{-1/5}$, and the first inequality of the result follows. The upper bound is obtained in the same way by letting $m = 1 - \eps$.
\end{enumerate}
\end{proof}

\subsubsection{Second-order comparison}

Let us now provide second-moment estimates for the random walk $(S_n)$. This relies on similar but slightly more tedious computations as before, and we shall thus not provide all the details.

\begin{lemma} \label{lem:moddevSn2}
As $n \to \pinf$,
\[
\left | \E \left (  g \left ( |\tS_n| \right ) g \left ( |\tS_n(\th)| \right ) \right ) - \E \left (  g \left ( |\tB_n| \right ) g \left ( |\tB_n| \right ) \right ) \right | = \frac{1}{n^{2 \a}} O \left( \frac1n \right ).
\]
In particular, for a fixed $\eps_n$ such that $\log n / n \ll \eps_n \ll 1$,
\begin{align*}
\left | \E \left ( g \left ( |\tS_n| \right ) g \left ( |\tS_n(\th)| \right ) \right ) \right. & \left. -  \E \left ( g \left ( |\tS_n(\th)| \right ) \right ) \E \left ( g \left ( |\tS_n(\th)| \right ) \right ) \right | \\
& = \frac{1}{n^{2 \a}} \left ( O \left ( \frac1n \right ) +  O \left ( \frac{\log n}{n \th} \right ) \right )
\end{align*}
uniformly for $\th \in [\eps_n,1)$.
\end{lemma}

\begin{proof}
Let us consider, for $\th \in [0,1)$,
\[
f(z_1,\dots,z_n) = \frac{1}{\phi(n)} (z_1 + \dots + z_n, z_1 e^{2 i \pi \th} + \dots + z_n e^{2 i \pi n \th}), \quad g(z,z') = p_{1,2}(z) p_{1,2}(z').
\]
We write for $z \in \C$, $\th \in [0,1)$,
\[
[\bfX,\bfY](\th)_j = X_1 e^{2 i \pi \th} + \dots + X_{j-1} e^{2 i \pi (j-1) \th} + Y_{j+1} e^{2 i \pi (j+1) \th} + \dots + Y_n e^{2 i \pi n \th}.
\]
It is easy to check that for some universal constant $C$, in the notation of Lemma \ref{lem:invprinciple} and the previous section,
\[
\begin{split}
\E(R_j) & \leq \frac{C}{(\eps \phi(n))^4} \left ( \, \P (|[\bfX,\bfY]_j| > \phi(n) - K_n, |[\bfX,\bfY](\th)_j| > \phi(n) - K_n) \right ) \\
& \quad + \P(|X_j| > K_n) + \P(|Y_j| > K_n)
\end{split}
\]
and once again, we shall bound the first probability with Bernstein's inequality. First, using the same trick as above, write
\begin{align*}
\P (|[\bfX,\bfY]_j| > & \phi(n) - K_n, |[\bfX,\bfY](\th)_j| > \phi(n) - K_n ) \\
 & \leq \sum_{k,l = 1}^{d_n} \P ([\bfX,\bfY]_j \cdot e^{2 i k \pi / d_n} > \psi(n), [\bfX,\bfY](\th)_j \cdot e^{2 i l \pi / d_n} > \psi(n) ) \\
 & = d_n \sum_{k = 1}^{d_n} \P ([\bfX,\bfY]_j \cdot 1 > \psi(n), [\bfX,\bfY](\th)_j \cdot e^{2 i k \pi / d_n} > \psi(n) ) \\
 & \leq d_n \sum_{k = 1}^{d_n} \P \left ( \frac12 ( [\bfX,\bfY]_j \cdot 1 + [\bfX,\bfY](\th)_j \cdot e^{2 i k \pi / d_n} ) > \psi(n) \right )
\end{align*}
where the penultimate step stems from the rotational invariance. We may then rewrite
\begin{align*}
\frac12 \left ( |[\bfX,\bfY]_j| \cdot 1 + |[\bfX,\bfY](\th)_j| \cdot e^{2 i k \pi / d_n} \right ) & = \frac12 \sum_{r=1}^{j-1} ( C_r ( 1 + \cos ( 2 \pi ( r \th + k / d_n ) ) ) \\
& + C'_r \sin ( 2 \pi ( r \th + k / d_n ) ) ) \\
& + \frac12 \sum_{r=1}^{j-1} ( N_r ( 1 + \cos ( 2 \pi ( r \th + k / d_n ) ) ) \\
& + N'_r \sin ( 2 \pi ( r \th + k / d_n ) ) ) \\
& := \sum_{r=1}^{j-1} A_r + \sum_{r=j+1}^n A_r,
\end{align*}
where the variables $A_r$ are independent. One readily checks that
\[
\E(A_r) = 0, \quad \E(A_r^2) = \frac12 \left ( 1 + \cos( 2 \pi (r \th + k / d_n) ) \right ),
\]
and
\[
\sum_{r=1}^n \cos( 2 \pi (r \th + k / d_n) ) = \frac{\sin n \pi \th}{\sin \pi \th} \cos ( \pi ( 2 k / d_n + (n-1) \th ) )
\]
so that, uniformly for $\th \in [\eps_n,\pi]$,
\[
\sum_{r=1}^{j-1} \E(A_r^2) + \sum_{r=j+1}^n \E(A_j^2) \leq \frac{n}{2} \left ( 1 + O \left ( \frac{1}{n \th} \right ) \right ) = \frac{n}{2} \left ( 1 + O \left ( \frac{1}{n \eps_n} \right ) \right ).
\]
Then, as before, Bernstein's inequality and truncation imply that
\begin{align*}
\P \left ( \frac12 ( |[\bfX,\bfY]_j| \cdot 1 \right. & \left. {}+ \vphantom{\frac12} |[\bfX,\bfY](\om)_j| \cdot e^{2 i k \pi / d_n} ) > \psi(n) \right ) \\
& \leq \exp - \frac{\psi(n)^2/2}{\frac{n}{2} \left ( 1 + O \left ( \frac{1}{n \eps_n} \right ) \right ) + \psi(n) K_n / 3} + C n e^{- \k K_n} \\
& \leq C \frac{1}{n^{2 \a}}.
\end{align*}
Gathering the pieces, we get
\[
| \E ( f \circ g (\bfX)) - \E (f \circ g (\bfY)) | \leq \frac{1}{n^{2 \a}} O \left ( \frac1n \right ).
\]
The second part of the result is just given by using as well Lemma \ref{lem:moddev} and \ref{lem:moddevSn1}.
\end{proof}

\subsection{Angular deviations} \label{sec:angdev2}

In pretty much the same fashion as Lemma \ref{lem:angdev}, we may prove the following.

\begin{lemma} \label{lem:angdevSn}
Fix $\beta > 0$, a sequence $\eps_n \ll n^{-(1+\b)}$ and $\eta > 0$. Then there exists a sequence $K_n$, diverging to $\pinf$, depending only on $\eps_n$ and $\eta$, such that
\[
\P \left ( \sup_{ | \th - \th' | < \eps_n} |S_n(\th) - S_n(\th')| > \eta \phi(n) \right ) = O \left ( n^{- K_n } \right ).
\]
\end{lemma}

\begin{proof}
Just as in the proof of Lemma \ref{lem:angdev}, and with the same notation, we may take $\th'=0$ and write
\begin{align*}
\P \left ( \vphantom{\sup_{\th \in [0,\eps_n]}} \right. & \left. \sup_{\th \in [0,\eps_n]} |S_n(\th) - S_n| > \eta \phi(n) \right ) \\
& \leq  \sum_{j=1}^{k-1} \P \left ( C \left | \sum_{r=1}^n r^j U_j \right | \geq \frac1k \eta \phi(n) \right ) + \sum_{r=1}^n \P \left ( C n^k \eps_n^k |U_r| \geq \frac{1}{k n} \eta \phi(n) \right ).
\end{align*}
The first term can be dealt with by writing
\begin{align*}
\P \left ( C \left | \sum_{r=1}^n r^j U_j \right | \geq \frac{\eta \phi(n)}{k} \right ) & \leq \P \left ( C \left | \sum_{r=1}^n r^j \Re(U_j) \right | \geq \frac{\eta \phi(n)}{2k} \right ) + \P \left ( C \left | \sum_{r=1}^n r^j \Im(U_j) \right | \geq \frac{\eta \phi(n)}{2k} \right ) \\
& = 4 \P \left ( C \sum_{r=1}^n r^j \Re(U_j) \geq \frac{\eta \phi(n)}{2k} \right ),
\end{align*}
where we use the rotational invariance, and then using Bernstein's inequality, which shows that it tends to 0 faster than any power of $n$; one may also use the same trick as in Section \ref{sec:moddevSn1}. The second term is easy to bound using the fact that $|U_r|$ has an exponential moment, and the result follows immediately.
\end{proof}

\subsection{End of the proof}

We shall now construct a tree as in Section \ref{sec:treedim}. Fix $q > 1$, and for $v = v_i^n \in \cT$, define
\[
Z_v = p_{1,2} \left ( |\tS_{q^n}(i 2^{-n})| \right ).
\]
Let also
\[
D = \{ \th \in [0,1), \# \{ v \in R(\th), Z_v > 0 \} = \pinf \}
\]
the limsup fractal associated to $(Z_v)$. We shall prove the equivalent of Proposition \ref{prop:treedim}, namely the following result.

\begin{prop} \label{prop:treedim2}
Almost surely
\[
(1 - \a) \log_2 q \leq \dim D_{\a} \leq (1 - \a) \vee 0
\]
for $q \in (1,2)$, and
\[
\dim D_{\a} =  (1 - \a \log_2 q) \vee 0
\]
for $q > 2$.
\end{prop}

\begin{proof}
\begin{enumerate}[wide]

\item The upper bounds are obtained as for Proposition \ref{prop:treedim}. Note indeed that, for $|v| = n$,
\[
\P(Z_v > 0) \leq \P \left ( |S_n| > \s \phi(n) \right ) = O ( q^{-n \a}
\]
according to Lemma \ref{lem:moddevSn1}. The remaining of the proof is similar as for Proposition \ref{prop:treedim}, using Lemmas \ref{lem:upbound} and \ref{lem:angdevSn}.

\item To get the lower bound, let
\[
m_n = \E(Z_v), \quad |v|=n.
\]
Then Lemmas \ref{lem:moddevSn1} and \ref{lem:moddevSn2} provide
\[
m_n = \frac{1}{q^{\a n}} \left ( 1 + O \left ( q^{-n} \right ) \right )
\]
and
\[
\left | \E(Z_v Z_u) - \E(Z_v) \E(Z_u) \right | = \frac{1}{q^{2 \a n}} \left ( O \left ( q^{-n} \right ) +  O \left ( \frac{n}{q^n |v-u|} \right ) \right )
\]
for $v \neq u$. We are thus able to do the same computations as in the proof of Proposition \ref{prop:treedim}, so as to apply Lemma \ref{lem:upbound} and pick $\zeta(n) = O (1)$, which provides
\[
\dim D \geq 1 - \a \log_2 q.
\]
\end{enumerate}
\end{proof}

The remaining of the proof of Theorem \ref{th:dimSn}, as one would expect, is just copy-pasting the end of the proof of \eqref{eq:dimBn}, using the result just proven and Lemma \ref{lem:angdevSn}. The only difference is that one needs to bound, for $1 < q < 2$, $r \in \{ q^n + 1, \dots, q^{n+1} \}$,
\[
P(|S_r(i 2^{-n}) - S_{q^n}(i2^{-n})| > \eta \phi(q^n)
\]
which is readily done as in the proof of Lemma \ref{lem:moddevSn1}, by truncating at level $\log^2 n$ and using Bernstein's inequality.

\bibliographystyle{abbrv}
\bibliography{Bibli}

\end{document}